\documentclass[reqno]{amsart}
\usepackage[english]{babel}
\usepackage{amscd,amssymb,amsmath,amsfonts,latexsym,amsthm}
\usepackage{inputenc}
\usepackage{graphicx, tikz-cd}

\numberwithin{equation}{section}

\numberwithin{equation}{section}

\newtheorem{thm}{Theorem}[section]
 \newtheorem{cor}[thm]{Corollary}
 
 \newtheorem{prop}[thm]{Proposition}
 \newtheorem{defn}[thm]{Definition}

 \newtheorem{ex}[thm]{Example}
\def\vtr{\vartriangleright}
\def\vtl{\vartriangleleft}

\begin{document}
\title[Classification of three dimensional anti-dendriform algebras]{Classification of three dimensional anti-dendriform algebras}
		
\author{Abdurasulov K., Adashev J., Normatov Z.,  Solijonova Sh.}

\address[Kobiljon Abdurasulov]{
Universidade da Beira Interior, R. Marqu\^{e}s de \'{A}vila e Bolama, Covilh\~{a}, 6201-001,  Portugal; Institute of Mathematics, Uzbekistan Academy of Sciences, Univesity Street, 9, Olmazor district, Tashkent, 100174, Uzbekistan}
\email{abdurasulov0505@mail.ru}
\address[Jobir Adashev]{
Institute of Mathematics, Uzbekistan Academy of Sciences, Univesity Street, 9, Olmazor district, Tashkent, 100174, Uzbekistan		
}
\email{adashevjq@mail.ru}
\address[Zafar Normatov]{Institute of Mathematics, Uzbekistan Academy of Sciences, Univesity Street, 9, Olmazor district, Tashkent, 100174, Uzbekistan}
\email{z.normatov@mathinst.uz}
\address[Shokhsanam Solijonova]{	
National University of Uzbekistan, Univesity Street, 4, Olmazor district, Tashkent, 100174, Uzbekistan}
\email{sh.solijonova@mail.ru}
	
\begin{abstract} This article is devoted to the classification of anti-dendriform algebras that are associated with associativity. They are characterized as algebras with two operations whose sum is associative. In particular, the paper is devoted to classifying anti-dendriform algebras associated with null-filiform associative algebras and three-dimensional algebras. It is known that any finite-dimensional associative algebra without a non-zero idempotent element is nilpotent. If an associative algebra has a non-zero idempotent element, then there does not exist a compatible anti-dendriform algebra structure associated with associative algebras.
\end{abstract}

	
	\maketitle

\section{Introduction}
\

The idea is to start with two distinct operations for the product $xy$ and the product $yx$ so that the bracket is not necessarily skew-symmetric anymore. Explicitly, we define an associative dialgebra as a vector space equipped with two associative operations \cite{Ch,Fra,U}. In \cite{bai20} introduced a dual notion of the Poisson algebra by exchanging the roles of the two binary operations in the Leibniz rule defining the Poisson algebra. Moreover, it was shown that the transposed Poisson algebra thus defined not only shares common properties of the Poisson algebra, including the closure under taking tensor products and the Koszul self-duality as an operad, but also admits a rich class of identities. In paper \cite{GP} authors defined a dendriform di- or trialgebra in an arbitrary variety Var of binary algebras (associative, commutative, Poisson, etc.) and proved that every dendriform dialgebra can be embedded into a Rota--Baxter algebra of weight zero in the same variety, and every dendriform trialgebra can be embedded into a Rota--Baxter algebra of nonzero weight. In these papers \cite{Le1,LR,RRB} also study algebras that are defined by several multiplications and require satisfying identities.

The notion of anti-dendriform algebras as a new approach of splitting the associativity is introduced in \cite{DGC}. Anti-dendriform algebras are characterized as algebras with two operations whose sum is associative and the negative left and right multiplication operators compose the bimodules of the sum associative algebras, justifying the notion due to the comparison with the corresponding characterization of dendriform algebras.

Recall that a dendriform algebra is a vector space $A$ with two bilinear operations $\succ,\prec$ satisfying
$$
x\succ(y\succ z)=(x\succ y+x\prec y)\succ z, \ (x\prec y)\prec z=x\prec(y\succ z+y\prec z),\  (x\succ y)\prec z=x\succ (y\prec z),
$$
for all $x,y,z\in A$. The notion of dendriform algebras was introduced by Loday in \cite{Cha1}. The fact that the sum of the two operations in a
dendriform algebra $(A,\succ,\prec)$ gives an associative algebra
$(A,\cdot)$ expresses a kind of ``splitting the associativity".
 Moreover, dendriform algebras are closely related to
pre-Lie algebras which are a class of Lie-admissible algebras
whose commutators are Lie algebras, also appearing in many fields in
mathematics and physics (\cite{Bai2,Bu} and the references
therein), in the sense that for a dendriform algebra
$(A,\succ,\prec)$, the bilinear operation
$$
x\ast  y=x\succ y-y\prec x,\;\;\forall x,y\in A,
$$
defines a pre-Lie algebra $(A,\ast)$, which is called the
associated pre-Lie algebra of $(A,\succ,\prec)$. Therefore there
is the following relationship among Lie algebras,
associative algebras, pre-Lie algebras and dendriform algebras in
the sense of commutative diagram of categories (\cite{Cha1}):
$$ \begin{matrix} {\rm dendriform\quad algebras} & \longrightarrow & \mbox{pre-Lie algebras} \cr \downarrow & &\downarrow\cr {\rm
associative\quad algebras} & \longrightarrow & {\rm Lie\quad
algebras.} \cr\end{matrix}$$

Note that there is an anti-structure for pre-Lie algebras, namely anti-pre-Lie algebras, introduced in \cite{LB}, which are characterized as the Lie-admissible algebras. There is a new approach to splitting operations, motivated by the study of anti-pre-Lie algebras. We introduce the notion of anti-dendriform algebras, still keeping the property of splitting the associativity, but it is the negative left and right multiplication operators that compose the bimodules of the sum associative algebras, instead of the left and right multiplication operators doing so for dendriform algebras. Such a characterization justifies the notion, and the following commutative diagram holds, which is the above diagram with replacing dendriform and pre-Lie algebras by anti-dendriform and anti-pre-Lie algebras respectively.

$$\begin{matrix} \mbox{ anti-dendriform algebras} & \longrightarrow & \mbox{anti-pre-Lie algebras} \cr \downarrow & &\downarrow\cr {\rm
associative\quad algebras} & \longrightarrow & {\rm Lie\quad
algebras} \cr\end{matrix}$$

The classification of any class of algebras is a fundamental and very difficult problem. It is one of the first problems that one encounters when trying to understand the structure of this class of algebras. This paper is devoted to classifying anti-dendriform algebras associated with null-filiform associative algebras and three-dimensional algebras. The algebraic classification (up to isomorphism) of algebras of dimension $n$ from a certain variety defined by a certain family of polynomial identities is a classic problem in the theory of non-associative algebras. There are many results related to the algebraic classification of small-dimensional algebras in many varieties of
associative and non-associative algebras \cite{ak21,afk21,BM,CKLS,fkk21,fkkv22,FKL,gkp21,ikp20,Kay}.
So, algebraic classifications of $2$-dimensional algebras \cite{petersson}, $3$-dimensional evolution algebras \cite{ccsmv}, $3$-dimensional anticommutative algebras \cite{Kobayashi,japan}, $3$-Dimensional diassociative algebras \cite{RIB} and classification of non-isomorphic complex $3$-dimensional transposed Poisson algebras \cite{BOK},  have been given.

The paper is organized as follows. In Section 2, we give some necessary notations, definitions, and preliminary results. In Section 3, we prove that there is not compatible anti-dendriform algebra structure on null-filiform associative algebras. We obtain all classifications of three-dimensional anti-dendriform algebras in Section 4.

\section{Preliminaries}

Let $A$ be a vector space with two bilinear operations
\[
\vtr: A\otimes A \rightarrow A, \qquad \vtl: A\otimes A \rightarrow A.
\]
Define a bilinear operation $\cdot$ as
\begin{equation}\label{cdot}
x \cdot y=x \vtr y+ x\vtl y, \quad \forall \ x, y \in A.
\end{equation}
If $(A,\cdot)$ is an associative algebra, then we call the triple $(A, \vtr, \vtl)$ an \textit{associative admissible algebra} and  $(A,\cdot)$ \textit{the associated associative algebra of} $(A, \vtr, \vtl)$.

\begin{defn}
Let $A$ be a vector space with two bilinear operations $\vartriangleright$ and $\vartriangleleft$. The triple $(A, \vtr, \vtl)$ is called an anti-dendriform algebra if the following equations hold:
\begin{equation}\label{anti1}
x\vtr(y\vtr z)=-(x\cdot y)\vtr z=-x\vtl (y\cdot z)=(x\vtl y)\vtl z,
\end{equation}
\begin{equation}\label{anti2}
(x\vtr y)\vtl z=x\vtr(y\vtl z), \quad \forall \ x,y,z\in A.
\end{equation}
\end{defn}

\begin{prop}[\cite{DGC}]
Let $(A, \vtr, \vtl)$ be an anti-dendriform algebra. Define a bilinear operation $\cdot$ by \eqref{cdot}.
Then $(A, \cdot)$ is an associative algebra, called
\textbf{the associated associative algebra} of $(A, \vtr, \vtl)$. Furthermore, $(A, \vtr, \vtl)$ is called
\textbf{a compatible anti-dendriform algebra structure} on $(A, \cdot)$.
\end{prop}

Now we divide equations \eqref{anti1} and \eqref{anti2} to some parts which we use throughout the paper:
\begin{align}\label{id1}
(x\vtr y)\vtl z=x\vtr(y\vtl z),\\\label{id2}
x\vtr (y\vtr z)=-(x\cdot y)\vtr z,\\\label{id3}
x\vtr (y\vtr z)=-x\vtl(y\cdot z),\\\label{id4}
x\vtr (y\vtr z)=(x\vtl y)\vtl z,\\\label{id5}
(x\cdot y)\vtr z=x\vtl(y\cdot z),\\\label{id6}
-(x\cdot y)\vtr z=(x\vtl y)\vtl z,\\\label{id7}
-x\vtl (y\cdot z)=(x\vtl y)\vtl z.
\end{align}

Recall that an associative algebra $(A, \cdot)$ is 2-nilpotent if $(x\cdot y)\cdot z = x\cdot (y\cdot z) = 0$ for all $x, y, z \in A$.
\begin{defn}
    An anti-dendriform algebra $(A,\vtr,\vtl)$ is 2-nilpotent if $(x\ast_{i_1} y)\ast_{i_2} z=x\ast_{i_3} (y\ast_{i_4} z)=0$ for all $i_1, i_2, i_3, i_4\in\{\vtr,\vtl\}$ and $x,y,z\in A$.
\end{defn}

Then the following conclusion is obvious.

\begin{prop}\label{2-nilpotent}
Let $(A,\vtr,\vtl)$ be an anti-dendriform algebra satisfying $x\vtr y+x\vtl y=0$ for all $x,y\in A$.  Then $(A,\vtr,\vtl)$ is 2-nilpotent algebra.
\end{prop}

\begin{proof}
Let $x,y,z\in A$. Then using Eq. \eqref{anti1}, we obtain
 $$x\vtr(y\vtr z)\stackrel{(\ref{id2})}{=}-(x\vtr y+x\vtl y)\vtr z=0=-x\vtl(y\vtr z+y\vtl y)\stackrel{(\ref{id7})}{=}(x\vtl y)\vtl z.$$

Moreover, we have \begin{equation}\label{eq1}
x\vtr(y\vtr z)=(x\vtl y)\vtl z=0.
\end{equation}
By Eq. (\ref{eq1}) and Eq. (\ref{anti2}), we have
\begin{equation}\label{eq2}
(x\vtr y)\vtl z=(x\vtr y)\vtl z+(x\vtl y)\vtl z=(x\vtr y+x\vtl y)\vtl z=0,
\end{equation}
\begin{equation}\label{eq3}
x\vtr (y\vtl z)=x\vtr (y\vtr z)+x\vtr (y\vtl z)=x\vtr(y\vtr z +y\vtl z)=0.
\end{equation}
Furthermore, by Eqs.~(\ref{eq1}), (\ref{eq2}) and (\ref{eq3}),
we have
\begin{equation*}
(x\vtr y)\vtr z=(x\vtr y)\vtr z+(x\vtr y)\vtl z=0,
\end{equation*}
\begin{equation*}
(x\vtl y)\vtr z=(x\vtl y)\vtr z+(x\vtl y)\vtl z=0,
\end{equation*}
\begin{equation*}
x\vtl (y\vtr z)=x\vtr (y\vtr z)+x\vtl (y\vtr z)=0,
\end{equation*}
\begin{equation*}
x\vtl (y\vtl z)=x\vtr (y\vtl z)+x\vtl (y\vtl z)=0.
\end{equation*}
\end{proof}

The following sets are called the center of associative and anti-dendriform algebras, respectively: $${\rm
Z}_{As}(A)=\{x\in A~|~x\cdot y=y\cdot x=0,\ \forall y\in A\},$$
$${\rm
Z}_{AD}(A)=\{x\in A~|~x\vtr y=x\vtl y=y\vtr x=y\vtl x=0,\ \forall y\in A\}.$$

An \textit{ideal} $I$ of a anti-dendriform algebra $A$ is a subalgebra of the algebra $A$ that satisfies the conditions:
$$x\vtr y, \ x\vtl y, \ y\vtr x, \ y\vtl x \in I, \ \ \mbox{for any} \ \  x\in A,\  y\in I.$$
It is obvious that center of an arbitrary anti-dendriform algebra is an ideal.

\begin{prop}\label{compatible}
Let $(A,\cdot)$ be an associative algebra and let $(A, \vtr, \vtl)$ be a compatible anti-dendriform algebra structure on $(A,\cdot)$. If the center $(Z(A), \cdot)$ of $(A, \cdot)$ is the center $(Z(A), \vtr, \vtl)$ of $(A, \vtr, \vtl)$, then the quotient $(A/Z(A), \vtr, \vtl)$ is a compatible anti-dendriform algebra structure on $(A/Z(A), \cdot)$.
\end{prop}

\begin{proof}
The proposition can be described as follows
$$
\begin{CD}
(A, \vtr, \vtl) @>\quad\quad x\cdot y=x\vtr y+x\vtl y\quad\quad>> (A, \cdot)  \\
@V{\bar{x}=x+Z(A)}VV @VV{\bar{x}=x+Z(A)}V  \\
(A/Z(A), \vtr, \vtl) @>>\quad\quad\bar{x}\cdot \bar{y}=\bar{x}\vtr \bar{y}+\bar{x}\vtl \bar{y}\quad\quad> (A/Z(A), \cdot)
\end{CD}
$$
To show that $(Z(A), \vtr, \vtl)$ is a compatible anti-dendriform algebra structure on $(A/Z(A), \cdot)$, it is sufficient to show $\bar{x}\cdot \bar{y}=\overline{x\cdot y}$.  The structure is given by $\bar{x}\cdot\bar{y}=\bar{x}\vtr \bar{y}+\bar{x}\vtl\bar{y}$. Indeed,
$$
\begin{array}{lll}
\bar{x}\cdot\bar{y}&=&\bar{x}\vtr \bar{y}+\bar{x}\vtl\bar{y}\\
&=&(x+Z(A))\vtr(y+Z(A))+(x+Z(A))\vtl(y+Z(A))\\
&=&x\vtr y+x\vtl y+Z(A)=\overline{x\cdot y}.
\end{array}
$$
\end{proof}

From now we use the following notations: $As_n^q$- and $AD_n^q$- stand for $n$-dimensional $q$-th associative and anti-dendriform algebra structures associated with nilpotent associative algebras, respectively.

\begin{thm}[\cite{Kobayashi}] Let $A$ be a two-dimensional complex associative algebra. Then
it is isomorphic to one of the following pairwise non-isomorphic associative
algebras:

$As_2^1 : \ Abelian;$

$As_2^2 : \ e_1e_1 = e_1;$

$As_2^3 : \ e_1e_1= e_2;$

$As_2^4 : \ e_1e_1 = e_1, e_1e_2= e_2;$

$As_2^5 : \ e_1e_1 = e_1, e_2e_1= e_2;$

$As_2^6 : \ e_1e_1 = e_1, e_1e_2= e_2, e_2e_2= e_2;$

$As_2^7 : \ e_1e_1 = e_1, e_2e_2= e_2.$
\end{thm}

Now we present the classification of all two-dimensional anti-dendriform algebras corresponding to these two-dimensional associative algebras.

\begin{thm}[\cite{DGC}] Let $(A, \rhd,\lhd )$ be a compatible two-dimensional complex anti-dendriform algebra structure on $(A, \cdot)$ two-dimensional complex associative algebra. Then any two-dimensional complex anti-dendriform algebra $(A, \rhd, \lhd )$ is isomorphic to one of the following mutually non-isomorphic cases:

$AD_2^1: \ Abelian;$

$AD_2^2: \ e_1\lhd e_1 = e_2;$

$AD_2^3(\lambda): \ e_1\rhd e_1 = e_2, e_1\lhd e_1= \lambda e_2, \ \lambda\in \mathbb{C}.$

\end{thm}

In this theorem, it can be seen that the compatible anti-dendriform algebras $AD_2^1$ and $AD_2^3(-1)$ structure on $As_2^1$,  the anti-dendriform algebras $AD_2^2$ and $AD_2^3(\lambda_{\neq-1})$ are compatible anti-dendriform algebra structure on $As_2^3$.

\begin{prop}[\cite{DGC}]\label{idempotent}
Let $(A, \cdot)$ be an associative algebra with a non-zero idempotent $e$, that is, $e \cdot e= e$. Then there does not exist a compatible anti-dendriform algebra structure on $(A, \cdot)$.
\end{prop}

It is known that any finite-dimensional associative algebra without a non-zero idempotent element is nilpotent. Therefore we have the following conclusion.

\begin{cor}
The associated associative algebra of any anti-dendriform algebra is nilpotent.
\end{cor}

For an algebra $A$ of an arbitrary variety, we consider the series
$$A^1=A, \ \ A^{i+1}=\sum\limits_{k=1}^{i}A^kA^{i+1-k}, \ \ i\geq1.$$

We say that an algebra $A$ is \textbf{nilpotent} if $A^i =\{0\}$ for some $i\in \mathbb{N}.$ The smallest integer satisfying $A^i = \{0\}$
is called the index of nilpotency of $A.$

According to this result, in order to classify arbitrary three-dimensional anti-dendriform algebras over a complex number field, we present a classification of all complex three-dimensional nilpotent associative algebras.

\begin{thm}[\cite{Kobayashi}]\label{assdim3}
Any three-dimensional complex  nilpotent associative algebra is
isomorphic to one of the following pairwise non-isomorphic
associative algebras:

$As_3^1:\ Abelian;$

$As_3^2:\ e_1e_2=e_3, \ e_2e_1=-e_3;$

$As_3^3:\ e_1e_1=e_3;$

$As_3^4:\ e_1e_2=e_3;$

$As_3^5(\lambda):\ e_1e_1=e_3, \ e_1e_2=\lambda e_3, \ e_2e_2=e_3, \ \lambda\in \mathbb{C};$

$As_3^6:\ e_1e_1=e_2, \ e_1e_2=e_3, \ e_2e_1=e_3.$

\end{thm}

\begin{ex}[\cite{DGC}]\label{example3} Let $(A, \cdot)$ be a 3-dimensional associative algebra with a basis $\{e_1, e_2, e_3\}$ whose
nonzero products are given by $As_3^6.$ By a straightforward computation, $(A, \rhd, \lhd)$ is a compatible anti-dendriform algebra structure on
$As_3^6$ with the following non-zero products:
\[
e_1\vtr e_1=\frac12e_2+\alpha e_3,\ e_1\vtl e_1=\frac12e_2-\alpha e_3,\ e_1\vtr e_2=e_2\vtl e_1=2e_3,\ e_2\vtr e_1=e_1\vtl e_2=-e_3,\]
where $\alpha\in \mathbb{C}$. Note that algebras can be isomorphic for different values of the parameter $\alpha$.

\end{ex}

\begin{defn} An $n$-dimensional algebra $A$ is called \textbf{null-filiform} if $\dim A^i = (n+1)-i,  \  1\leq i\leq n+1.$

\end{defn}

Note that an algebra has a maximum nilpotency index if and only if it is null-filiform. For
a nilpotent algebra, the condition of null-filiformity is equivalent to the condition that the algebra is one generated.
All null-filiform associative algebras were described in \cite{DO}:

\begin{thm} An arbitrary $n$-dimensional null-filiform associative algebra is isomorphic to the algebra:
$$\mu_n^0: \ e_ie_j=e_{i+j}, \ 2\leq i+j\leq n,$$
where $\{e_1, e_2, \dots, e_n\}$ is a basis of the algebra $\mu_n^0$ and the omitted products vanish.
\end{thm}

\section{Anti-dendriform algebra associated to Null-filiform algebra}

In \cite{DGC} the authors fully classified one and two-dimensional anti-dendriform algebras as examples and gave one example for the three-dimensional anti-dendriform algebra. In this section, we prove that there does not exist
a compatible anti-dendriform structure on null-filiform algebras of dimension more than four and we will give a complect classification for three-dimensional case.

\begin{thm}
Let $(A,\cdot)$  be a $n$-dimensional associative  algebra over the complex field $\mathbb{C}$ with a basis $\{e_1,e_2, \dots, e_n\}$ whose non-zero products are as follows:
\[
e_i\cdot e_j=e_{i+j}, \quad 2\leq i+j\leq n.
\]
If $n\geq4$, then there is not a compatible anti-dendriform algebra structure on $(A,\cdot)$.
\end{thm}

\begin{proof}
Suppose there is a compatible anti-dendriform algebra structure on $\mu_n^0$. Let $(A,\vtr,\vtl)$ be such algebra structure on $\mu_n^0$. By definition we have $e_{i+j}=e_i\vtr e_j+e_i\vtl e_j$. Thus we identify $e_i \vtl e_j$ in terms of $e_i\vtr e_j$. Set $e_i\vtr e_j=\sum\limits_{k=1}^n\alpha_{i,j}^ke_k$. Then we can write
\[
\begin{cases}
e_i\vtr e_j=\sum\limits_{k=1}^n\alpha_{i,j}^ke_k,\\
e_i\vtl e_j=-\sum\limits_{k=1}^n\alpha_{i,j}^ke_k+e_{i+j},& 1\leq i,j\leq n, \ i+j\leq n,\\
e_i\vtl e_j=-\sum\limits_{k=1}^n\alpha_{i,j}^ke_k,& 1\leq i,j\leq n, \ i+j> n.
\end{cases}
\]
From \eqref{id5} for the triples $(1,i,j)$ we get
\[
e_{i+1}\vtr e_j=e_1\vtl e_{i+j}=-e_1\vtr e_{i+j}+e_{i+j+1},\ \ \quad 1\leq i\leq n-1.
\]
Hence,
$$
\begin{array}{ll}
e_i\vtr e_j=-e_1\vtr e_{i+j-1}+e_{i+j},& 2\leq i\leq n, \ 1\leq j \leq n-i, \ \ i+j\leq n,\\
e_i\vtr e_j=-e_1\vtr e_{i+j-1},& 2\leq i\leq n, \ 1\leq j \leq n-i, \ \ i+j=n+1,\\
e_i\vtr e_j=0,& 2\leq i\leq n, \ 1\leq j \leq n-i, \ \ i+j>n+1.
\end{array}
$$
Consider \eqref{id5} for the triples $(i,1,1),\ (1,i,1), \ 1\leq i\leq n-1$:
\[
\begin{split}
 e_{i+1}\vtr e_1&=e_i\vtl e_2,\\
 e_{i+1}\vtr e_1&=e_1\vtl e_{i+1}=e_{i+2}-e_1\vtr e_{i+1},\\
 e_{i}\vtl e_2&=e_{i+1}-e_i\vtr e_2=e_{i+2}-(-e_1\vtr e_{i+1}+e_{i+2})=e_1\vtr e_{i+1}.\\
\end{split}
\]
and we obtain
\[ e_1\vtr e_i=e_1\vtl e_i=\frac12e_{i+1}, \quad 3\leq i\leq n-1, \quad e_1\vtr e_n=e_1\vtl e_n=0.
\]

Thus, we can write multiplication table as follows:
\[
\begin{cases}
e_1\vtr e_i=\sum\limits_{k=1}^n\alpha_{1,i}^ke_k, & 1\leq i \leq 2,\\
e_1\vtr e_i=\frac12e_{i+1}, & 3\leq i\leq n-1,\\
e_1\vtr e_n=0,\\
e_2\vtr e_1=e_3-\sum\limits_{k=1}^n\alpha_{1,2}^ke_k,\\
e_i\vtr e_1=\frac12e_{i+1}, & 3\leq i\leq n-1,\\

e_i\vtr e_j=\frac12e_{i+j}, & 2\leq i\leq n, \ 2\leq j\leq n-i, \ i+j\leq n\\
e_i\vtr e_j=0, & 2\leq i\leq n, \ 1\leq j\leq n, \ i+j>n\\
e_1\vtl e_i=e_{i+1}-\sum\limits_{k=1}^n\alpha_{1,i}^ke_k, & 1\leq i\leq 2\\
e_1\vtl e_i=\frac12e_{i+1}, & 3\leq i\leq n-1,\\
e_2\vtl e_1=\sum\limits_{k=1}^n\alpha_{1,2}^ke_k,\\
e_i\vtl e_1=\frac12e_{i+1}, & 3\leq i\leq n-1,\\
e_i\vtl e_j=\frac12e_{i+j}, & 2\leq i\leq n, \ 2\leq j\leq n-i, \  \ i+j\leq n,\\
e_i\vtl e_j=0, & 2\leq i\leq n, \  1\leq j\leq n,\  i+j> n.
\end{cases}
\]
Consider \eqref{id2} for the triple $(1,1,1)$.
\[
\begin{split}
e_1\vtr(e_1\vtr e_1)&=-e_2\vtr e_1,\\
\sum_{k=1}^n\alpha_{1,1}^k(e_1\vtr e_k)&=-e_3+\sum_{k=1}^n\alpha_{1,2}^ke_k,\\
\alpha_{1,1}^1\sum_{k=1}^n\alpha_{1,1}^ke_k+\alpha_{1,1}^2\sum_{k=1}^n\alpha_{1,2}^ke_k+\frac12\sum_{k=3}^{n-1}\alpha_{1,1}^ke_{k+1}&=-e_3+\sum_{k=1}^n\alpha_{1,2}^ke_k.
\end{split}
\]
From this we derive the following system of equalities:
\begin{equation}\label{sys1}
\begin{cases}
\alpha_{1,1}^1\alpha_{1,1}^1+\alpha_{1,1}^2\alpha_{1,2}^1=\alpha_{1,2}^1,\\
\alpha_{1,1}^1\alpha_{1,1}^2+\alpha_{1,1}^2\alpha_{1,2}^2=\alpha_{1,2}^2,\\
\alpha_{1,1}^1\alpha_{1,1}^3+\alpha_{1,1}^2\alpha_{1,2}^3=-1+\alpha_{1,2}^3,\\
\alpha_{1,1}^1\alpha_{1,1}^k+\alpha_{1,1}^2\alpha_{1,2}^k+\frac12\alpha_{1,1}^{k-1}=\alpha_{1,2}^k \ \ 4\leq k\leq n.
\end{cases}
\end{equation}
From \eqref{id4} for the triple $(1, 1, 1)$ we get
\begin{equation*}e_1\vtr (e_1\vtr e_1)=\sum_{k=1}^n\alpha_{1,1}^k(e_1\vtr e_k)=\end{equation*}
\begin{equation}\label{111left}=\alpha_{1,1}^1\sum_{k=1}^n\alpha_{1,1}^ke_k+\alpha_{1,1}^2\sum_{k=1}^n\alpha_{1,2}^ke_k+\frac12\sum_{k=3}^{n-1}\alpha_{1,1}^ke_{k+1}.
\end{equation}
On the other hand
\begin{equation*}(e_1\vtl e_1)\vtl e_1=-\sum_{k=1}^n\alpha_{1,1}^k(e_k\vtl e_1)+e_2\vtl e_1=\end{equation*}
\begin{equation}\label{111right}=\alpha_{1,1}^1\sum_{k=1}^n\alpha_{1,1}^ke_k-\alpha_{1,1}^1e_2-\alpha_{1,1}^2\sum_{k=1}^n\alpha_{1,2}^ke_k-\frac12\sum_{k=3}^{n-1}\alpha_{1,1}^ke_{k+1}+\sum_{k=1}^n\alpha_{1,2}^ke_k,\\
\end{equation}
Equating the expressions (\ref{111left}) and (\ref{111right}), we get the equality:
$$
2\alpha_{1,1}^2\sum_{k=1}^n\alpha_{1,2}^ke_k+\sum_{k=3}^{n-1}\alpha_{1,1}^ke_{k+1}=-\alpha_{1,1}^1e_2+\sum_{k=1}^n\alpha_{1,2}^ke_k.
$$
Comparing coefficients at the basis of this equality derives
\begin{equation}\label{sys2}
\begin{cases}
2\alpha_{1,1}^2\alpha_{1,2}^1=\alpha_{1,2}^1,\\
2\alpha_{1,1}^2\alpha_{1,2}^2=-\alpha_{1,1}^1+\alpha_{1,2}^2,\\
2\alpha_{1,1}^2\alpha_{1,2}^3=\alpha_{1,2}^3,\\
2\alpha_{1,1}^2\alpha_{1,2}^k+\alpha_{1,1}^{k-1}=\alpha_{1,2}^k, \qquad 4\leq k\leq n.
\end{cases}
\end{equation}
The first equality in \eqref{sys2} implies two possible cases: $\alpha_{1,2}^1\neq0$ or $\alpha_{1,2}^1= 0$.

\textbf{Case 1.} Let $\alpha_{1,2}^1\neq0$. Then $\alpha_{1,1}^2=\frac12$. The second equality in \eqref{sys2} implies $\alpha_{1,1}^1=0$. From \eqref{sys1} we get $\alpha_{1,2}^1=0$ which is a contradiction.

\textbf{Case 2.} Let $\alpha_{1,2}^1=0$. From \eqref{sys1} we get $\alpha_{1,1}^1=0$ and
\begin{equation}\label{sys3}
\begin{cases}
\alpha_{1,1}^2\alpha_{1,2}^2=\alpha_{1,2}^2,\\
\alpha_{1,1}^2\alpha_{1,2}^3=\alpha_{1,2}^3-1,\\
\alpha_{1,1}^2\alpha_{1,2}^k+\frac12\alpha_{1,1}^{k-1}=\alpha_{1,2}^k, \ \ 4\leq k\leq n.
\end{cases}
\end{equation}
From \eqref{sys2} and \eqref{sys3} we obtain the following  restrictions
  $$\alpha_{1,2}^2=0, \ \alpha_{1,2}^3=2, \  \alpha_{1,2}^k=0,  \ 4\leq k\leq n, \ \alpha_{1,1}^2=\frac12, \ \alpha_{1,1}^t=0,  \
 \ 3\leq t\leq n-1.$$
Then by considering \eqref{id1} for the triple $(1,1,2)$ we derive
$$(e_1\vtr e_1)\vtl e_2=e_1\vtr (e_1\vtl e_2),$$$$
(\frac12e_2+\alpha_{1,1}^ne_n)\vtl e_2=e_1\vtr (-e_3),$$$$
\frac14e_4=-\frac12e_4,
$$which implies a contradiction when there is a basis element $e_4$.
\end{proof}

\begin{cor}\label{null-fliform} Any three-dimensional anti-dendriform algebra associated to the null-fliform algebra $\mu_3^0$ is isomorphic to one of the two non-isomorphic algebras:
$$AD_3^1: \    e_1\vtr e_1=\frac12e_2, \ e_1\vtl e_1=\frac12e_2,\ e_1\vtr e_2=e_2\vtl e_1=2e_3,\ e_2\vtr e_1=e_1\vtl e_2=-e_3;$$
$$AD_3^2: \ e_1\vtr e_1=\frac12e_2+e_3,\ e_1\vtl e_1=\frac12e_2-e_3,\ e_1\vtr e_2=e_2\vtl e_1=2e_3, \ e_2\vtr e_1=e_1\vtl e_2=-e_3.$$
\end{cor}
\begin{proof} According to Case 1 in the proof of the previous theorem, there is no algebra, and due to Case 2 we obtain the multiplication table of three-dimensional anti-dendriform algebra associated to the associated on the null-filiform algebra $\mu_3^0$ as Example \ref{example3}:
$$e_1\vtr e_1=\frac12e_2+\alpha e_3,\ e_1\vtl e_1=\frac12e_2-\alpha e_3, \ e_1\vtr e_2=e_2\vtl e_1=2e_3, \ e_2\vtr e_1=e_1\vtl e_2=-e_3.$$

By using the multiplication of the null-fliform algebra we consider the general basis change:
$$e'_1=A_1e_1+A_2e_2+A_3e_3, \ e'_2=A_1^2e_2+2A_1A_2e_3, \ e'_3=A_1^3e_3.$$

Equality $e'_1\vtr e'_1=\frac12e'_2+\alpha' e'_3$ derives $$\alpha'=\frac1{A_1}\alpha.$$

If $\alpha=0,$ then we obtain the algebra $AD_3^1$.

If $\alpha\neq 0,$ then putting $a_1=\alpha$ we obtain the algebra $AD_3^2$.
\end{proof}


\section{Classification of three dimensional anti-dendriform algebras}

In this section we will classify all three-dimensional anti-dendriform algebras. According to Proposition \ref{idempotent} it is enough to consider only nilpotent algebras. According to Theorem \ref{assdim3}, there are six nilpotent algebras $As^1_3-As_3^6$. Since $As_3^6$ is a null-fliform algebra, we have already carried out the classification in Corollary \ref{null-fliform}.

\begin{thm} Any three-dimensional complex anti-dendriform algebra associated to the algebra $As^1_3$ is isomorphic to one of the following pairwise non-isomorphic algebras:

$AD_3^3: \ \mbox{is trivial, that is, all products are zero;}$

$AD_3^4: \ e_1\vtr e_2=e_3, \ e_2\vtr e_1=-e_3,\ e_1\vtl e_2=-e_3, \ e_2\vtl e_1=e_3;$

$AD_3^5: \ e_1\vtr e_1=e_3,\ e_1\vtl e_1=-e_3;$

$AD_3^6: \ e_1\vtr e_2=e_3, \ e_1\vtl e_2=-e_3;$

$AD_3^7(\lambda): \ \begin{cases}
e_1\vtr e_1=e_3,\ e_1\vtr e_2=\lambda e_3, \ e_2\vtr e_2=e_3,\\
e_1\vtl e_1=-e_3,\ e_1\vtl e_2=-\lambda e_3, \ e_2\vtl e_2=-e_3;
\end{cases}
$\\
where $\lambda\in \mathbb{C}.$
\end{thm}
\begin{proof} Let $(A, \cdot)$ be an abelian algebra $As_3^1$ and by definition $e_i\cdot e_j=e_i\vtr e_j+e_i\vtl e_j=0$. Then by Proposition \ref{2-nilpotent} the anti-dendriform algebra $(A, \vtr, \vtl)$ associated to $As_3^1$ is 2-nilpotent. Thus we have $(x\vtr y)\vtr z=x\vtr (y\vtr z)=0$ which implies that $(A, \vtr)$ is associative 2-nilpotent algebra. According to Theorem \ref{assdim3} there are five non-isomorphic three-dimensional associative 2-nilpotent algebras.

    It is possible to consider the multiplication $\vtl$ as above. However, it is not difficult to show that the constructed algebras are isomorphic. This finishes the proof.

\end{proof}

Now we will describe three-dimensional anti-dendriform associated to the algebra $As_3^2$ with the multiplication:
$$e_1e_2=e_3, \ e_2e_1=-e_3. $$

\begin{thm}
Any three-dimensional complex anti-dendriform algebra associated to the algebra $As_3^2$ is isomorphic to one of the following pairwise non-isomorphic algebras:

$AD_3^8(\alpha,\beta): \ \begin{cases}
e_1\vtr e_1=e_3, \ e_1\vtr e_2=\alpha e_3, \
e_2\vtr e_1=\beta e_3,\\[1mm]
e_1\vtl e_1=-e_3, \  e_1\vtl e_2=(1-\alpha)e_3, \
e_2\vtl e_1=(-1-\beta)e_3;
\end{cases}$\\[1mm]

$AD_3^9(\alpha): \ e_1\vtr e_2=\alpha e_3, \ e_2\vtr e_1=-\alpha e_3, \ e_1\vtl e_2=(1-\alpha)e_3, \ e_2\vtl e_1=(-1+\alpha)e_3;
$\\[1mm]

$AD_3^{10}: \ e_1\vtr e_1=e_2, \ e_2\vtr e_1=-e_3, \ e_1\vtl e_1=-e_2, \ e_1\vtl e_2=e_3;$\\[1mm]
where $\alpha, \beta\in \mathbb{C}$ and $AD_3^8(\alpha,\beta)\cong AD_3^8(-\beta,-\alpha)$.
\end{thm}
\begin{proof}
By considering \eqref{id5} for the following triples
$$\{e_1,e_2,e_3\},\ \{e_3,e_1,e_2\},\ \{e_1,e_1,e_2\}, \ \{e_1,e_2,e_2\},\ \{e_2,e_1,e_2\},\ \{e_1,e_2,e_1\}$$ we get
$e_3\vtr e_3=0, \ e_3\vtl e_3=0, \ e_1\vtl e_3=0, \ e_3\vtr e_2=0, \ e_2\vtl e_3=0, \ e_3\vtr e_1=0$, respectively.

By $e_3e_i=e_3\vtr e_i+e_3\vtl e_i=0$ and $e_ie_3=e_i\vtr e_3+e_i\vtl e_3=0$ these imply $e_3\vtl e_i=0$ and $e_i\vtr e_3=0$. Then it is easy to see that $\langle e_3\rangle$ is the center of the algebras $(As_3^2, \cdot)$ and
 $(As_3^2, \vtr, \vtl)$.

According to Proposition \ref{compatible} one can write
\[
\left\{\begin{array}{llll}
e_1\vtr e_1=\delta e_2+\alpha_{1}e_3, &e_1\vtr e_2=\alpha_{2}e_3, & e_2\vtr e_1=\alpha_{3}e_3, & e_2\vtr e_2=\alpha_{4}e_3,\\[1mm]
e_1\vtl e_1=-\delta e_2-\alpha_{1}e_3, & e_1\vtl e_2=(1-\alpha_{2})e_3, & e_2\vtl e_1=(-1-\alpha_{3})e_3, & e_2\vtl e_2=-\alpha_{4}e_3,
\end{array}\right.
\]
where $\delta\in\{0,1\}$

\textbf{Case $\delta=0$.}  Let us consider the general change of the generators of basis:
\[e'_1=A_1e_1+A_2e_2+A_3e_3, \quad e'_2=B_1e_1+B_2e_2+B_3e_3,\]
and from the product $e'_1e'_2=e'_3$ we put $e'_3=(A_1B_2-A_2B_1)e_3,$ where $(A_1B_2-A_2B_1)^2\neq0.$

We express the new basis elements $\{e'_1, e'_2, e'_3\}$  via the basis elements $\{e_1, e_2, e_3\}.$  By verifying all the multiplications of the algebra in the new basis we obtain the relations between the parameters $\{\alpha'_{1}, \alpha'_{2}, \alpha'_{3}, \alpha'_{4}\}$ and $\{\alpha_{1}, \alpha_{2}, \alpha_{3}, \alpha_{4}\}$:
\begin{equation}\label{zamena1}
\begin{array}{l}
\alpha'_1=\frac{A_1^2\alpha_{1}+A_1A_2(\alpha_{2}+\alpha_{3})+A_2^2\alpha_{4}}{A_1B_2-A_2B_1},\\[1mm] \alpha'_2=\frac{A_1B_1\alpha_{1}+A_1B_2\alpha_{2}+A_2B_1\alpha_{3}+A_2B_2\alpha_{4}}{A_1B_2-A_2B_1},\\[1mm] \alpha'_3=\frac{A_1B_1\alpha_{1}+A_2B_1\alpha_{2}+A_1B_2\alpha_{3}+A_2B_2\alpha_{4}}{A_1B_2-A_2B_1},\\[1mm] \alpha'_4=\frac{B_1^2\alpha_{1}+B_1B_2(\alpha_{2}+\alpha_{3})+B_2^2\alpha_{4}}{A_1B_2-A_2B_1}.\end{array}
\end{equation}
Then we have the following cases.
\begin{itemize}

\item[1.] Let $(\alpha_1,\alpha_4)\neq(0,0)$. Then without loss of generality, we will assume that $\alpha_1\neq0$.  Next, choosing $B_1=\frac{-(\alpha_2+\alpha_3)\pm\sqrt{(\alpha_2+\alpha_3)^2-4\alpha_1\alpha_4}}{2\alpha_1}B_2$, we can put $\alpha'_4=0$. Thus we obtained the following algebra:
\[
\left\{\begin{array}{lll}
e_1\vtr e_1=\alpha_{1}e_3, &e_1\vtr e_2=\alpha_{2}e_3, & e_2\vtr e_1=\alpha_{3}e_3,\\[1mm]
e_1\vtl e_1=-\alpha_{1}e_3, & e_1\vtl e_2=(1-\alpha_{2})e_3, & e_2\vtl e_1=(-1-\alpha_{3})e_3. \\
\end{array}\right.
\]

Again, using the change of basis the expressions (\ref{zamena1}) will take the following form:
$$\begin{array}{l}
\alpha'_1=\frac{A_1^2\alpha_{1}+A_1A_2(\alpha_{2}+\alpha_{3})}{A_1B_2-A_2B_1},\\[1mm] \alpha'_2=\frac{A_1B_1\alpha_{1}+A_1B_2\alpha_{2}+A_2B_1\alpha_{3}}{A_1B_2-A_2B_1},\\[1mm] \alpha'_3=\frac{A_1B_1\alpha_{1}+A_2B_1\alpha_{2}+A_1B_2\alpha_{3}}{A_1B_2-A_2B_1},\\[1mm] 0=B_1^2\alpha_{1}+B_1B_2(\alpha_{2}+\alpha_{3}).\end{array}$$

\begin{enumerate}
  \item If $B_1=0,$ applying $B_2=A_1\alpha_1+A_2(\alpha_2+\alpha_3),$ we get  $\alpha'_1=1, \ \alpha'_2=\alpha_2, \ \alpha'_3=\alpha_3$ and derive the following anti-dendriform algebra $AD_3^8(\alpha,\beta).$
  \item If $B_1\neq0,$ then we have $B_1\alpha_{1}+B_2(\alpha_{2}+\alpha_{3})=0$, and putting  $B_1=-(\alpha_2+\alpha_3)A_1$ and $B_2=A_1\alpha_1$, we get $\alpha'_1=1, \ \alpha'_2=-\alpha_3, \ \alpha'_3=-\alpha_2$ and derive the following anti-dendriform algebra $AD_3^8(\alpha,\beta).$
\end{enumerate}



By the change of basis $ e'_1=e_1, e'_2=-(\alpha+\beta)e_1+e_2, \ e'_3=e_3$ we can show that the algebras  $AD_3^8(\alpha,\beta)$ and $AD_3^8(-\beta, -\alpha)$ are isomorphic.

\item[2.] Let $\alpha_{1}=\alpha_{4}=0$ and $\alpha_{2}+\alpha_{3}\neq0$. Then by choosing nonzero values of $A_1$ and $A_2$, we can assume $\alpha'_1\neq 0$ which is Case 1.

 \item[3.] Let $\alpha_{1}=\alpha_{4}=0$  and $\alpha_{2}+\alpha_{3}=0.$ Then we get $\alpha'_{1}=\alpha'_{4}=0,$ and  $\alpha'_{2}=-\alpha'_{3}=\alpha_{2}$ and derive the anti-dendriform algebra $AD_3^9(\alpha).$
\end{itemize}

\textbf{Case $\delta=1$.} We can write
\[
\left\{\begin{array}{llll}
e_1\vtr e_1=e_2+\alpha_{1}e_3, &e_1\vtr e_2=\alpha_{2}e_3, & e_2\vtr e_1=\alpha_{3}e_3, & e_2\vtr e_2=\alpha_{4}e_3,\\[1mm]
e_1\vtl e_1=-e_2-\alpha_{1}e_3, & e_1\vtl e_2=(1-\alpha_{2})e_3, & e_2\vtl e_1=(-1-\alpha_{3})e_3, & e_2\vtl e_2=-\alpha_{4}e_3.
\end{array}\right.
\]

Using \eqref{id2} for the triple $\{e_1,e_1,e_1\}$ and \eqref{id6} for the triple $\{e_1,e_1,e_1\}$,  $\{e_1,e_1,e_2\}$ we obtain $\alpha_2=0,$ $\alpha_3=-1,$ $\alpha_4=0$. Now we consider the basis change $e'_2=e_2+\alpha_1e_3$. Then we obtain the algebra $AD_3^{10}:$
$$e_1\vtr e_1=e_2, \ e_2\vtr e_1=-e_3, \ e_1\vtl e_1=-e_2, \ e_1\vtl e_2=e_3.$$ \end{proof}

Now, we consider the algebra $As_3^4$ with the multiplication $e_1e_2=e_3$.

\begin{thm}
Any three-dimensional complex anti-dendriform algebra associated to the algebra $As_3^4$ is isomorphic to the following pairwise  non-isomorphic algebras:

$AD_3^{11}(\alpha,\beta): \
e_1\vtr e_2=\alpha e_3,\
e_2\vtr e_1=\beta e_3,\
e_1\vtl e_2=(1-\alpha)e_3,\
e_2\vtl e_1=-\beta e_3;
$

$AD_3^{12}(\alpha,\beta): \ \begin{cases}
e_1\vtr e_2=\alpha e_3,\
e_2\vtr e_1=\beta e_3,\
e_2\vtr e_2=e_3,\\
e_1\vtl e_2=(1-\alpha)e_3,\
 e_2\vtl e_1=-\beta e_3,\
e_2\vtl e_2=-e_3;
\end{cases}$

$ AD_3^{13}(\alpha,\beta,\gamma): \ \begin{cases}
e_1\vtr e_1=e_3, \
e_1\vtr e_2=\alpha e_3,\
e_2\vtr e_1=\beta e_3,\
e_2\vtr e_2=\gamma e_3,\\
e_1\vtl e_1=-e_3,\
e_1\vtl e_2=(1-\alpha)e_3,\
e_2\vtl e_1=-\beta e_3,\
e_2\vtl e_2=-\gamma e_3;
\end{cases}$

$AD_3^{14}:\ e_1\vtr e_1=e_2, \ e_1\vtl e_1=-e_2, \ e_1\vtl e_2=e_3;$\\
where $\alpha, \beta, \gamma\in \mathbb{C}.$
\end{thm}
\begin{proof}
By considering \eqref{id5} for the triples $\{e_1,e_2,e_3\}$, $\{e_1,e_2,e_\}$, $\{e_1,e_2,e_1\}$, $\{e_2,e_1,e_2\}$, $\{e_1,e_1,e_2\}$ we get
$e_3\vtr e_3=0, \ \ e_3\vtr e_2=0, \ \ e_3\vtr e_1=0, \ \ e_2\vtl e_3=0, \ \ e_1\vtl e_3=0$, respectively. By \eqref{cdot} one can get $e_3\vtl e_3=0, \ \ e_3\vtl e_2=0, \ \ e_3\vtl e_1=0, \ \ e_2\vtr e_3=0, \ \ e_1\vtr e_3=0$.

By $e_3e_i=e_3\vtr e_i+e_3\vtl e_i=0$ and $e_ie_3=e_i\vtr e_3+e_i\vtl e_3=0$ these imply $e_3\vtl e_i=0$ and $e_i\vtr e_3=0$. Then it is easy to see that $\langle e_3\rangle$ is the center of the algebras $(As_3^4, \cdot)$ and
 $(As_3^4, \vtr, \vtl)$.

According to Proposition \ref{compatible} we can write
\[
\begin{cases}
e_1\vtr e_1=\delta e_2+\alpha_1e_3, \
e_1\vtr e_2=\alpha_2e_3,\
e_2\vtr e_1=\alpha_3e_3,\
e_2\vtr e_2=\alpha_4e_3,\\
e_1\vtl e_1=-\delta e_2-\alpha_1e_3,\
e_1\vtl e_2=(1-\alpha_2)e_3,\
e_2\vtl e_1=-\alpha_3e_3,\
e_2\vtl e_2=-\alpha_4e_3,
\end{cases}
\]
where $\delta\in\{0,1\}$.

\textbf{Case $\delta=0$.} Let us consider the general change of the generators of basis:
\[e'_1=A_1e_1+A_2e_2+A_3e_3, \quad e'_2=B_1e_1+B_2e_2+B_3e_3,\]
and from the products $e'_1e'_2=e'_3, \ e'_1e'_1=0, \ e'_2e'_2=0,$ we put $e'_3=A_1B_2e_3, \ A_1A_2=0, \ B_1B_2=0,$ where $A_1B_2(A_1B_2-A_2B_1)\neq0.$ This implies $A_2=B_1=0$ and $A_1B_2\neq0$.

We express the new basis elements $\{e'_1, e'_2, e'_3\}$  via the basis elements $\{e_1, e_2, e_3\}.$  By verifying all the multiplications of the algebra in the new basis we obtain the relations between the parameters $\{\alpha'_{1}, \alpha'_{2}, \alpha'_{3}, \alpha'_{4}\}$ and $\{\alpha_{1}, \alpha_{2}, \alpha_{3}, \alpha_{4}\}$:
$$
\alpha'_1=\frac{A_1\alpha_{1}}{B_2},\ \alpha'_2=\alpha_{2},\
\alpha'_3=\alpha_{3},\
\alpha'_4=\frac{B_2\alpha_{4}}{A_1}$$
with condition $A_2=0$.
Then we have the following cases:

\begin{enumerate}
  \item If $\alpha_1=\alpha_4=0$. Then we have the algebra $AD_3^{11}(\alpha,\beta)$.

\item If $\alpha_1=0$ and $\alpha_4\neq0$. Then by choosing $A_1=B_2\alpha_4$ we can obtain the algebra $AD_3^{12}(\alpha,\beta)$.

 \item If $\alpha_1\neq0$. Then putting $B_2=A_1\alpha_1$ we get the algebra $AD_3^{13}(\alpha,\beta,\gamma)$.

\end{enumerate}

\textbf{Case $\delta=1$.} We can write
\[
\begin{cases}
e_1\vtr e_1=e_2+\alpha_1e_3, \
e_1\vtr e_2=\alpha_2e_3,\
e_2\vtr e_1=\alpha_3e_3,\
e_2\vtr e_2=\alpha_4e_3,\\
e_1\vtl e_1=-e_2-\alpha_1e_3,\
e_1\vtl e_2=(1-\alpha_2)e_3,\
e_2\vtl e_1=-\alpha_3e_3,\
e_2\vtl e_2=-\alpha_4e_3.
\end{cases}
\]
Using \eqref{id2} for the triples $\{e_1,e_1,e_1\}$, $\{e_2,e_1,e_1\}$ and \eqref{id6} for the triple $\{e_1,e_1,e_1\}$ we obtain $\alpha_2=\alpha_3=\alpha_4=0$. Finally, by considering the basis change $e'_2=e_2+\alpha_1e_3$, we obtain the algebra $AD_3^{14}$. \end{proof}

Let us consider the algebra $As_3^5(\lambda)$ with the multiplication
\[ e_1e_1=e_3, \ e_1e_2=\lambda e_3, \  e_2e_2=e_3. \]

\begin{thm}
Any three-dimensional complex anti-dendriform algebra associated to the algebra $As_3^5(\lambda)$
is isomorphic to the following pairwise non-isomorphic algebras:

$AD_3^{15}(\alpha,\beta,\gamma,\lambda):\ \begin{cases}
e_1\vtr e_1=\alpha e_3,\
e_1\vtr e_2=\beta e_3,\
e_2\vtr e_1=\gamma e_3, \  \alpha\neq 0,\\
e_1\vtl e_1=(1-\alpha)e_3,\
e_1\vtl e_2=(\lambda-\beta)e_3,\
e_2\vtl e_1=-\gamma e_3,\
e_2\vtl e_2=e_3;
\end{cases}$

$AD_3^{16}(\alpha, \lambda):\ \begin{cases}
e_1\vtr e_2=\alpha e_3,\
e_2\vtr e_1=-\alpha e_3,\\
e_1\vtl e_1=e_3,\
e_1\vtl e_2=(\lambda-\alpha)e_3,\
e_2\vtl e_1=-\alpha e_3,\
e_2\vtl e_2=e_3;
\end{cases}$

$AD_3^{17}(\lambda): \ e_1\vtr e_1=e_2,\
e_1\vtl e_1=-e_2+e_3,\
e_1\vtl e_2=\lambda e_3,\ e_2\vtl e_2=e_3;
$\\[1mm]
where  $\alpha, \beta, \gamma, \lambda\in \mathbb{C}$ and $AD_3^{15}(\alpha,\beta,\gamma,0)\cong AD_3^{15}(\alpha,-\beta,-\gamma,0).$

\end{thm}
\begin{proof}
By considering \eqref{id5} for the triples $$\{e_2,e_2,e_3\}, \{e_2,e_2,e_1\}, \ \{e_1,e_1,e_1\}, \  \{e_1,e_2,e_2\},\ \{e_2,e_2,e_2\}$$ we get
$e_3\vtr e_3=e_3\vtr e_1=e_1\vtl e_3=e_3\vtr e_2=e_2\vtl e_3=0$, respectively.

By $e_3e_i=e_3\vtr e_i+e_3\vtl e_i=0$ and $e_ie_3=e_i\vtr e_3+e_i\vtl e_3=0$ these imply $e_3\vtl e_i=0$ and $e_i\vtr e_3=0$. Then it is easy to see that $\langle e_3\rangle$ is the center of the algebras $(As_3^4, \cdot)$ and
 $(As_3^4, \vtr, \vtl)$.

According to Proposition \ref{compatible} we can write
\[
\begin{cases}
e_1\vtr e_1=\delta e_2+\alpha_1e_3,\
e_1\vtr e_2=\alpha_2e_3,\
e_2\vtr e_1=\alpha_3e_3,\
e_2\vtr e_2=\alpha_4e_3,\\
e_1\vtl e_1=-\delta e_2+(1-\alpha_1)e_3,\
e_1\vtl e_2=(\lambda-\alpha_2)e_3,\
e_2\vtl e_1=-\alpha_3e_3,\
e_2\vtl e_2=(1-\alpha_4)e_3,
\end{cases}
\]
where $\delta\in\{0,1\}$.

\textbf{Case $\delta=0$.} Let us consider the general change of the generators of basis:
\[e'_1=A_1e_1+A_2e_2+A_3e_3, \  e'_2=B_1e_1+B_2e_2+B_3e_3,\]
and from the products $e'_1e'_2=\lambda e'_3, \ e'_1e'_1=e'_3, \ e'_2e'_2=e'_3,\ e'_2e'_1=0,$ we put
$$e'_3=(A_1^2+\lambda A_1A_2+A_2^2)e_3, \ A_1B_2-A_2B_1\neq0, \ A_1B_1+A_2B_2=-\lambda A_2B_1,$$
$$A_1^2+\lambda A_1A_2+A_2^2=B_1^2+\lambda B_1B_2+B_2^2, \  \lambda (A_1^2+\lambda A_1A_2+A_2^2)=\lambda(A_1B_2-A_2B_1).$$

We express the new basis elements $\{e'_1, e'_2, e'_3\}$  via the basis elements $\{e_1, e_2, e_3\}.$  By verifying all the multiplications of the algebra in the new basis we obtain the relations between the parameters $\{\alpha'_{1}, \alpha'_{2}, \alpha'_{3}, \alpha'_{4}\}$ and $\{\alpha_{1}, \alpha_{2}, \alpha_{3}, \alpha_{4}\}$.

$$
\alpha'_1=\frac{A_1^2\alpha_{1}+A_1A_2(\alpha_2+\alpha_3)+A_2^2\alpha_4}{A_1^2+\lambda A_1A_2+A_2^2},\
\alpha'_2=\frac{A_1B_1\alpha_{1}+A_1B_2\alpha_2+A_2B_1\alpha_3+A_2B_2\alpha_4}{A_1^2+\lambda A_1A_2+A_2^2},$$
$$\alpha'_3=\frac{A_1B_1\alpha_{1}+A_2B_1\alpha_2+A_1B_2\alpha_3+A_2B_2\alpha_4}{A_1^2+\lambda A_1A_2+A_2^2},\  \alpha'_4=\frac{B_1^2\alpha_{1}+B_1B_2(\alpha_2+\alpha_3)+B_2^2\alpha_4}{A_1^2+\lambda A_1A_2+A_2^2}.
$$

Then we have the following cases.
\begin{itemize}

\item[1.] Let $(\alpha_1,\alpha_4)\neq0$. Without loss of generality, we assume that $\alpha_1\neq0$.  Next, choosing $B_1=\frac{-(\alpha_2+\alpha_3)\pm\sqrt{(\alpha_2+\alpha_3)^2-4\alpha_1\alpha_4}}{2\alpha_1}B_2$, we can put $\alpha'_4=0$. Thus we obtained the following algebra:
\[
\left\{\begin{array}{lll}
e_1\vtr e_1=\alpha_{1}e_3, &e_1\vtr e_2=\alpha_{2}e_3, & e_2\vtr e_1=\alpha_{3}e_3,\\[1mm]
e_1\vtl e_1=-\alpha_{1}e_3, & e_1\vtl e_2=(1-\alpha_{2})e_3, & e_2\vtl e_1=(-1-\alpha_{3})e_3. \\
\end{array}\right.
\]
Again by using a change of basis we obtain the following relations:
$$
\alpha'_1=\frac{A_1^2\alpha_{1}+A_1A_2(\alpha_2+\alpha_3)}{A_1^2+\lambda A_1A_2+A_2^2},\
\alpha'_2=\frac{A_1B_1\alpha_{1}+A_1B_2\alpha_2+A_2B_1\alpha_3}{A_1^2+\lambda A_1A_2+A_2^2},$$
$$
\alpha'_3=\frac{A_1B_1\alpha_{1}+A_2B_1\alpha_2+A_1B_2\alpha_3}{A_1^2+\lambda A_1A_2+A_2^2},\
B_1^2\alpha_{1}+B_1B_2(\alpha_2+\alpha_3)=0.
$$

\begin{enumerate}
  \item If $B_1=0.$ Then we have $A_1B_2\neq 0,\ A_2=0, \ \lambda (A_1-B_2)=0, \ (A_1-B_2)(A_1+B_2)=0$ and
$$\alpha'_1=\alpha_{1},\  \alpha'_2=\frac{B_2\alpha_2}{A_1},\  \alpha'_3=\frac{B_2\alpha_3}{A_1},$$
 we derive the anti-dendriform algebra $AD_3^{15}(\alpha,\beta,\gamma,\lambda)_{\alpha\neq0}.$ If $\lambda=0$ then $AD_3^{15}(\alpha,\beta,\gamma,0)_{\alpha\neq0}\cong AD_3^{15}(\alpha,-\beta, -\gamma,0)_{\alpha\neq0}.$

  \item If $B_1\neq0.$ Then we get $B_1=-\frac{B_2(\alpha_{2}+\alpha_{3})}{\alpha_1}\neq0$.
  \begin{enumerate}
    \item $\lambda=0.$ Then $A_1=\frac{A_2\alpha_1}{\alpha_2+\alpha_3},\ A_2B_2^2(\alpha_1^2+(\alpha_2+\alpha_3)^2)\neq0, \ B_2^2=\frac{A_2^2\alpha_1^2}{(\alpha_{2}+\alpha_{3})^2}$ and
    $$
\alpha'_1=\alpha_1,\ \alpha'_2=-\frac{B_2(\alpha_2+\alpha_3)\alpha_3}{A_2\alpha_1},\ \alpha'_3=-\frac{B_2(\alpha_2+\alpha_3)\alpha_2}{A_2\alpha_1}.$$
Thus, we have the anti-dendriform algebra $AD_3^{15}(\alpha,\beta, \gamma,0)_{\alpha\neq0, \beta+\gamma\neq0}.$
    \item $\lambda\neq0.$ Then we obtain $A_1=\frac{A_2(\alpha_1-\lambda(\alpha_2+\alpha_3))}{\alpha_2+\alpha_3},$ $A_2(\alpha_1^2-\lambda\alpha_1(\alpha_2+\alpha_3)+(\alpha_2+\alpha_3)^2)\neq0, \ B_2^2=\frac{A_2^2\alpha_1^2}{(\alpha_2+\alpha_3)^2}$ and
    $$
\alpha'_1=\alpha_1-\lambda(\alpha_2+\alpha_3),\
\alpha'_2=-\frac{B_2(\alpha_2+\alpha_3)\alpha_3}{A_2\alpha_1},\
\alpha'_3=-\frac{B_2(\alpha_2+\alpha_3)\alpha_2}{A_2\alpha_1}.$$
   In this case we derive the anti-dendriform algebra $AD_3^{15}(\alpha,\beta,\gamma,\lambda)_{\alpha\neq0, \beta+\gamma\neq0}.$
  \end{enumerate}
\end{enumerate}

\item[2.] Let $\alpha_{1}=\alpha_{4}=0$ and $\alpha_{2}+\alpha_{3}\neq0$. By choosing nonzero values of $A_1$ and $A_2$, we can assume $\alpha'_1\neq 0$ which is Case 1.

 \item[3.] Let $\alpha_{1}=\alpha_{4}=0$  and $\alpha_{2}+\alpha_{3}=0.$ Then
 \begin{enumerate}
   \item If $\lambda\neq0$ we have $AD_3^{16}(\alpha, \lambda)$.
   \item If $\lambda=0$ we have $A_1B_1+A_2B_2=0$ and $A_1^2+A_2^2=B_1^2+B_2^2.$ Thus we obtian $AD_3^{16}(\alpha, 0).$
 \end{enumerate}
\end{itemize}

\textbf{Case $\delta=1$.} We can write
\[
\begin{cases}
e_1\vtr e_1=e_2+\alpha_1e_3,\
e_1\vtr e_2=\alpha_2e_3,\
e_2\vtr e_1=\alpha_3e_3,\
e_2\vtr e_2=\alpha_4e_3,\\
e_1\vtl e_1=-e_2+(1-\alpha_1)e_3,\
e_1\vtl e_2=(\lambda-\alpha_2)e_3,\
e_2\vtl e_1=-\alpha_3e_3,\
e_2\vtl e_2=(1-\alpha_4)e_3.
\end{cases}
\]
Using \eqref{id2} for the triples $\{e_1,e_1,e_1\}$, $\{e_2,e_1,e_1\}$ and \eqref{id6} for the triple $\{e_1,e_1,e_1\}$ we obtain $\alpha_2=\alpha_3=\alpha_4=0$. Further, by changing of basis $e'_2=e_2+\alpha_1e_3$ we obtain the algebra $AD_3^{17}(\lambda).$ \end{proof}

Finally we explore classification of three-dimensional anti-dendriform associated to the algebra $As_3^3$.

\begin{thm}
Any three-dimensional complex anti-dendriform algebra associated to $As^3_3$ is isomorphic to one of the following pairwise non-isomorphic algebras:

$AD_3^{18}(\alpha): \ e_1\vtr e_1=\alpha e_3, \ e_1\vtl e_1=(1-\alpha)e_3;$

$AD_3^{19}: \ e_2\vtr e_1=e_3,\ e_1\vtl e_1=e_3, \ e_2\vtl e_1=-e_3;$

$AD_3^{20}(\alpha): \ \begin{cases}
e_1\vtr e_1=\alpha e_3, \ e_1\vtr e_2=e_3, \ e_2\vtr e_1=-e_3,\\
e_1\vtl e_1=(1-\alpha)e_3, \
e_1\vtl e_2=-e_3, \
e_2\vtl e_1=e_3;
\end{cases}$

$AD_3^{21}(\alpha):\  \begin{cases}
e_1\vtr e_2=e_3, \ e_2\vtr e_1=\alpha e_3,\\
e_1\vtl e_1=e_3, \
e_1\vtl e_2=-e_3, \
e_2\vtl e_1=-\alpha e_3, \ \alpha\neq-1;
\end{cases}$

$AD_3^{22}(\alpha,\beta): \ \begin{cases}
e_1\vtr e_1=\alpha e_3,  \ e_2\vtr e_1=\beta e_3, \  e_2\vtr e_2=e_3,\\
e_1\vtl e_1=(1-\alpha)e_3, \
e_2\vtl e_1=-\beta e_3, \
e_2\vtl e_2=-e_3;
\end{cases}$

$AD_3^{23}:\ e_1\vtr e_1=e_2,\ e_1\vtl e_1=-e_2+e_3;$\\
where $AD_3^{21}(-1)\cong AD_3^{20}(0).$

\end{thm}
\begin{proof}
By considering \eqref{id5} for the triples $\{e_1,e_1,e_3\}$, $\{e_1,e_1,e_2\}$, $\{e_3,e_1,e_1\}$, $\{e_2,e_1,e_1\}$ we get
$e_3\vtr e_3=0, \ \ e_3\vtr e_2=0, \ \ e_3\vtl e_3=0, \ \ e_2\vtr e_3=0$, respectively. By definition \eqref{cdot} one can get $e_3\vtl e_2=0, \ \ e_2\vtr e_3=0$ and $e_3\vtr e_1=-e_3\vtl e_1=e_1\vtl e_3=-e_1\vtr e_3$.

We get the following notations:
\[
\begin{cases}
e_i\vtr e_j=\sum\limits_{k=1}^3\alpha_{i,j}^ke_k, & 1\leq i,j\leq 2,\\
e_1\vtl e_1=e_3-\sum\limits_{k=1}^3\alpha_{1,1}^ke_k,\\
e_i\vtl e_j=-\sum\limits_{k=1}^3\alpha_{i,j}^ke_k, & 1\leq i,j\leq 2,\ (i,j)\neq(1,1),\\
e_3\vtr e_1=e_1\vtl e_3=\sum\limits_{k=1}^3\alpha_{3,1}^ke_k,\\
e_3\vtl e_1=e_1\vtr e_3=-\sum\limits_{k=1}^3\alpha_{3,1}^ke_k,\\
\end{cases}
\]
with considering \eqref{id5}
\[
\begin{split}
(e_1\vtr e_3)\vtr e_3=0 \quad &\Rightarrow \quad \alpha_{3,1}^1=0, \\
(e_1\vtr e_3)\vtr e_2=0 \quad &\Rightarrow \quad \alpha_{3,1}^2\alpha_{2,2}^1=\alpha_{3,1}^2\alpha_{2,2}^2=\alpha_{3,1}^2\alpha_{2,2}^3=0.
\end{split}
\]
Thus we have two cases: $\alpha_{3,1}^2\neq 0$ and $\alpha_{3,1}^2=0$.

\textbf{Case $\alpha_{3,1}^2\neq0$.} Then we get $\alpha_{2,2}^1=\alpha_{2,2}^2=\alpha_{2,2}^3=0$.

We have also
\[
\begin{split}
e_2\vtr (e_2\vtr e_1)=0 \quad &\Rightarrow \quad \alpha_{2,1}^1=0,\\
(e_1\vtl e_2)\vtl e_2=0 \quad &\Rightarrow \quad \alpha_{1,2}^1=0,\\
e_1\vtr (e_3\vtr e_1)=0 \quad \text{and} \quad (e_1\vtl e_3)\vtl e_1=0 \quad &\Rightarrow \quad e_1\vtr e_2=-e_2\vtr e_1.
\end{split}
\]
Now consider $e_1\vtr (e_1\vtr e_1)=-e_3\vtr e_1$, which derives
\[
\begin{split}
\sum_{k=1}^3\alpha_{1,1}^k(e_1\vtr e_k)&=-e_3\vtr e_1.
\end{split}
\]
Since there is not $e_1$ in the right hand side, we obtain $\alpha_{1,1}^1=0$.

Consider the \eqref{id4} and \eqref{id6} for the triple $\{e_1,e_1,e_1\}$, we have $e_3\vtl e_1=0.$ This implies $\alpha_{3,1}^2=0$ which is contradiction. Hence, there is not such case.

\textbf{Case $\alpha_{3,1}^2=0$}. We can obtain
\[
e_1\vtr (e_3\vtr e_1)=0 \qquad \Rightarrow \qquad \alpha_{3,1}^3=0.
\]
This derives $e_3\vtr e_1=0$ and as a result, by $e_3e_i=e_3\vtr e_i+e_3\vtl e_i=0$ and $e_ie_3=e_i\vtr e_3+e_i\vtl e_3=0$ this implies $e_3\vtl e_i=0$ and $e_i\vtr e_3=0$. Then it is easy to see that $\langle e_3\rangle$ is the center of the algebras $(As_3^4, \cdot)$ and
 $(As_3^4, \vtr, \vtl)$.

According to Proposition \ref{compatible} we can write
\[
\begin{cases}
e_1\vtr e_1=\delta e_2+\alpha_1e_3, \ e_1\vtr e_2=\alpha_2e_3, \ e_2\vtr e_1=\alpha_3e_3, \ e_2\vtr e_2=\alpha_4e_3,\\
e_1\vtl e_1=-\delta e_2+(1-\alpha_1)e_3, \
e_1\vtl e_2=-\alpha_2e_3, \
e_2\vtl e_1=-\alpha_3e_3, \
e_2\vtl e_2=-\alpha_4e_3,
\end{cases}
\]
where $\delta\in\{0,1\}.$
We note that hereafter we use the notation $(\alpha, \beta, \gamma, \lambda)$ for the algebra similar to the above algebra
since we work only with parameters.

\textbf{Case $\delta=0$.} Let us consider the general change of the generators of basis:
\[e'_1=A_1e_1+A_2e_2+A_3e_3, \quad e'_2=B_1e_1+B_2e_2+B_3e_3,\]
and from the products $e'_1e'_1=e'_3,\ e'_2e'_2=0$ we put $e'_3=A_1^2e_3, \ B_1=0,$ where $A_1^3B_2\neq0.$

We express the new basic elements $\{e'_1, e'_2, e'_3\}$  via the basic elements $\{e_1, e_2, e_3\}.$  By verifying all the multiplications of the algebra in the new basis we obtain the relations between the parameters $\{\alpha'_{1}, \alpha'_{2}, \alpha'_{3}, \alpha'_{4}\}$ and $\{\alpha_{1}, \alpha_{2}, \alpha_{3}, \alpha_{4}\}$:
$$\alpha'_1=\frac{A_1^2\alpha_{1}+A_1A_2(\alpha_2+\alpha_3)+A_2^2\alpha_4}{A_1^2},\ \alpha'_2=\frac{B_2(A_1\alpha_{2}+A_2\alpha_{4})}{A_1^2},\ \alpha'_3=\frac{B_2(A_1\alpha_{3}+A_2\alpha_{4})}{A_1^2},\ \alpha'_4=\frac{B_2^2\alpha_{4}}{A_1^2}.$$

Then we have the following cases.

\begin{enumerate}
    \item Let $\alpha_4=0$. Then we have
$$
\alpha'_1=\frac{A_1\alpha_1+A_2(\alpha_{2}+\alpha_{3})}{A_1}, \ \alpha'_2=\frac{B_2}{A_1}\alpha_2, \ \alpha'_3=\frac{B_2}{A_1}\alpha_3, \ \alpha'_4=0.$$

\begin{enumerate}
 \item If $\alpha_2=\alpha_3=0$. Then we have $
\alpha'_1=\alpha_1, \ \alpha'_2=\alpha'_3=\alpha'_4=0.$ Hence, we get the algebra $AD_3^{18}(\alpha)$.

 \item If $\alpha_2=0$ and $\alpha_3\neq0$. Then by choosing $B_2=\frac{A_1}{\alpha_3}, \ A_2=\frac{A_1\alpha_1}{\alpha_3}$ we obtain the algebra $AD_3^{19}$.

 \item If $\alpha_2\neq0$ and $\alpha_2+\alpha_3=0$. Then by putting $A_1=B_2\alpha_2$ we get the algebra $AD_3^{20}(\alpha)$.

\item If $\alpha_2\neq0$ and $\alpha_2+\alpha_3\neq0$. Then by choosing $A_2=-\frac{A_1\alpha_1}{\alpha_2+\alpha_3}$ we derive $\alpha'_1=0$. Furthermore by choosing suitable value of $B_2$ one can derive the algebra $AD_3^{21}(\alpha)$ where  $\alpha\neq-1$.
\end{enumerate}

\item Let $\alpha_4\neq0$. Then by choosing $B_2=\frac{A_1}{\sqrt{\alpha_4}}, \ A_2=-\frac{A_1\alpha_2}{\alpha_4}$ we derive the algebra $AD_3^{22}(\alpha,\beta)$.

\end{enumerate}
\textbf{Case $\delta=1$.} We can write
\[
\begin{cases}
e_1\vtr e_1=e_2+\alpha_1e_3, \ e_1\vtr e_2=\alpha_2e_3, \  e_2\vtr e_1=\alpha_3e_3, \  e_2\vtr e_2=\alpha_4e_3,\\
e_1\vtl e_1=-e_2+(1-\alpha_1)e_3, \
e_1\vtl e_2=-\alpha_2e_3, \
e_2\vtl e_1=-\alpha_3e_3, \
e_2\vtl e_2=-\alpha_4e_3.
\end{cases}
\]

Using \eqref{id2} for the triples $\{e_1,e_1,e_1\}$, $\{e_2,e_1,e_1\}$ and \eqref{id6} for the triple $\{e_1,e_1,e_1\}$ we obtain $\alpha_2=\alpha_3=\alpha_4=0$. Now we consider the basis change $e'_2=e_2+\alpha_1e_3$. Then we obtain
$AD_3^{23}.$

\end{proof}

\begin{center}
    \textbf{Acknowledgments}
\end{center}

The work is supported by FCT UIDB/MAT/00212/2020,  UIDP/MAT/00212/2020 and by grant F-FA-2021-423, Ministry of Higher Education, Science and Innovations of the Republic of Uzbekistan.

\end{document}